\documentclass{alggeom}

\usepackage[pdfstartview=FitH,
            colorlinks,
            linkcolor=reference,
            citecolor=citation,
            urlcolor=e-mail,
            ]{hyperref}

\usepackage{amsmath}
\usepackage{amscd}
\usepackage{amsbsy}
\usepackage{amssymb}
\usepackage{verbatim}
\usepackage{eufrak}
\usepackage{microtype}
\usepackage{hyperref}
\usepackage{mathrsfs}
\usepackage{amsthm}
\usepackage[all,cmtip]{xy}
\usepackage{tikz}
\usetikzlibrary{matrix,arrows}
\usepackage[abbrev,lite]{amsrefs}
\usepackage{youngtab}

\usepackage{mathpazo}

\usepackage{color}
\definecolor{todo}{rgb}{1,0,0}
\definecolor{conditional}{rgb}{0,1,0}
\definecolor{e-mail}{rgb}{0,.40,.80}
\definecolor{reference}{rgb}{.20,.60,.22}
\definecolor{mrnumber}{rgb}{.80,.40,0}
\definecolor{citation}{rgb}{0,.40,.80}

\usepackage{array}
\newcolumntype{M}[1]{>{\centering\arraybackslash}m{#1}}
\usepackage{listings}
\usepackage{lmodern}
\DeclareMathAlphabet{\mathpzc}{OT1}{pzc}{m}{it}
\usepackage{dsfont}

\DeclareMathOperator{\Gal}{Gal}
\DeclareMathOperator{\PGL}{PGL}
\DeclareMathOperator{\BSL}{BSL}

\DeclareMathOperator{\SL}{SL}
\DeclareMathOperator{\GL}{GL}

\DeclareMathOperator{\Spec}{Spec}

\DeclareMathOperator{\cl}{cl}

\DeclareMathOperator{\Hoh}{H}
\DeclareMathOperator{\Roh}{R}
\DeclareMathOperator{\Eoh}{E}

\DeclareMathOperator{\Sq}{Sq}
\renewcommand{\P}{\mathrm{P}}
\newcommand{\Q}{\mathrm{Q}}

\newcommand{\et}{\mathrm{\acute{e}t}}

\newcommand{\B}{\mathrm{B}}

\DeclareFontEncoding{OT2}{}{}

\newcommand{\CH}{\mathrm{CH}}

\renewcommand{\cl}{\mathrm{cl}}

\DeclareMathOperator{\K}{K}

\DeclareMathOperator{\KU}{KU}

\newcommand{\G}{\mathrm{G}}

\newcommand{\iso}{\cong}

\newcommand{\CC}{\mathds{C}}

\newcommand{\QQ}{\mathds{Q}}

\newcommand{\ZZ}{\mathds{Z}}

\newcommand{\FF}{\mathds{F}}
\newcommand{\PP}{\mathds{P}}

\newcommand{\Gm}{\mathds{G}_{m}}

\let\oldmarginpar\marginpar
\renewcommand\marginpar[1]{\-\oldmarginpar[\raggedleft\footnotesize #1]%
{\raggedright\footnotesize #1}}

\newcommand{\SO}{\mathrm{SO}}

\newcommand{\BG}{\mathrm{BG}}

\newcommand{\BPGL}{\mathrm{BPGL}}
\newcommand{\BGL}{\mathrm{BGL}}

\DeclareMathOperator{\chr}{char}

\theoremstyle{plain}
\newtheorem{theorem}{Theorem}[section]
\newtheorem*{theorem*}{Theorem}

\newtheorem{proposition}[theorem]{Proposition}

\newtheorem{conjecture}[theorem]{Conjecture}

\newtheoremstyle{named}{}{}{\itshape}{}{\bfseries}{.}{.5em}{#1 \thmnote{#3}}
\theoremstyle{named}

\theoremstyle{definition}

\theoremstyle{remark}

\begin{document}

\title{On the integral Tate conjecture for finite fields and representation theory}

\author{Benjamin Antieau}
\email{benjamin.antieau@gmail.com}
\address{University of Illinois at Chicago}


\begin{abstract}
    \noindent
    We describe a new source of counterexamples to the so-called integral Hodge and integral Tate
    conjectures. As in the other known counterexamples to the integral Tate conjecture over
    finite fields, ours are approximations of
    the classifying space of some group $\BG$. Unlike the other examples, we find groups
    of type $A_n$, our proof relies heavily on representation theory, and Milnor's
    operations vanish on the classes we construct.
\end{abstract}

\maketitle

\section{Introduction}

Recall that the original conjecture of Hodge predicted that the cycle class maps
\begin{equation*}
    \cl^i:\CH^i(X)\rightarrow\Hoh^{2i}(X,\ZZ)
\end{equation*}
would be surjective onto the subgroup of $(i,i)$-classes when $X$ is a smooth projective
variety. In~\cite{atiyah-hirzebruch} Atiyah and Hirzebruch constructed for each prime $\ell$
a smooth projective complex algebraic variety $X$ and an $\ell$-torsion integral cohomology
class $x\in\Hoh^4(X,\ZZ)$, necessarily of type $(2,2)$, such that $x$ is not in the image
of the cycle class map. They use the differentials $d_r$ in the Atiyah-Hirzebruch spectral
sequence and argue that $d_r$ vanishes on the image of $\cl^i$. By studying the cohomology of
finite groups, specifically $(\ZZ/\ell)^3$, they exhibit degree $4$ torsion classes not killed by these
differentials. Using K\"unneth, $\Hoh^*(\B(\ZZ/\ell)^3,\FF_\ell)$ has three degree $1$
generators, $y_1,y_2,y_3$. The class $y=\beta(y_1y_2y_3)\in\Hoh^4(\B(\ZZ/\ell)^3,\ZZ)$ is the
class used by Atiyah and Hirzebruch. They show that $d_{2\ell-1}(y)\neq 0$.
Any smooth projective variety that looks sufficiently like $\B(\ZZ/\ell)^3\times\B\Gm$ then
possesses a cohomology class $x\in\Hoh^4(X,\ZZ)$ with $d_{2\ell-1}(x)\neq 0$. Hence,
this class cannot be in the image of the cycle
class map. The Godeaux-Serre construction is one way to produce such varieties.

Tate's conjecture is for smooth projective varieties $X$ defined over a field $k$ finitely
generated over its prime field and
for a fixed prime $\ell$ different from $\chr(k)$. It asserts that the cycle class maps in
$\ell$-adic cohomology
\begin{equation*}
    \cl^i:\CH^i(X_{\overline{k}})_{\QQ_{\ell}}\rightarrow\bigcup_U\Hoh^{2i}(X_{\overline{k}},\QQ_{\ell}(i))^U\subseteq\Hoh^{2i}(X_{\overline{k}},\QQ_\ell(i))
\end{equation*}
are surjective, where $U$ varies over all open subgroups of $\Gal(\overline{k}/k)$ and
$\overline{k}$ is the algebraic closure of $k$.

There is an integral version of Tate's conjecture in which $\QQ_\ell(i)$ is replaced by
$\ZZ_\ell(i)$, but, like the integral Hodge conjecture, it is false and it appears that it
was not expected by Tate. Colliot-Th\'el\`ene and
Szamuely adapted the examples of Atiyah
and Hirzebruch to this situation in~\cite{colliot-thelene-szamuely}*{Th\'eor\`eme 2.1}. Non-torsion classes not in the image of the cycle
class map (either in singular cohomology over $\CC$ or in $\ell$-adic cohomology)
were constructed over uncountable fields by Koll\'ar using very general
hypersurface high-degree hypersurfaces in $\PP^4_\CC$.

Since the examples mentioned so far use either torsion cohomology classes or an
uncountable ground field,
Colliot-Th\'el\`ene and Szamuely asked
in~\cite{colliot-thelene-szamuely}*{Remark~2.2.3} whether a modified integral Tate
conjecture asserting the surjectivity of
\begin{equation}\label{eq:modtors}
    \CH^i(X_{\overline{k}})_{\ZZ_\ell}\rightarrow\Hoh^{2i}(X_{\overline{k}},\ZZ_\ell(i))/\mathrm{torsion}
\end{equation}
could still be true for smooth projective varieties over finite fields.

However, in~\cite{pirutka-yagita}, Pirutka and Yagita constructed counterexamples to the
surjectivity of~\eqref{eq:modtors} over finite
fields for $\ell=2,3,5$. Their examples are smooth projective
approximations to the classifying spaces $\B(\G\times\Gm)$ in characteristic $p$ for $\G$ a
simply connected affine algebraic group in the following cases:
$\G_2$ at the prime $\ell=2$, $\mathrm{F}_4$ at the prime $\ell=3$, and $\mathrm{E}_8$ at the prime $\ell=5$.
Each of these groups has $\Hoh^4(\BG,\ZZ)=\ZZ\cdot x$ for some generator $x$, and they show that $x$ is not
in the image of the cycle class map. This result gives new counterexamples to the integral Hodge
conjecture and, by using integral models and reduction modulo $p$, gives counterexamples to
the modulo-torsion integral Tate conjecture over finite fields, answering the question of
Colliot-Th\'el\`ene and Szamuely.

Shortly after the appearance of~\cite{pirutka-yagita},
Kameko~\cite{kameko} gave examples for all primes $\ell$ by studying
$\B(\SL_\ell\times\SL_\ell/\mu_\ell)$, where $\mu_\ell$ is embedded diagonally.

In all of these examples over finite fields, the argument actually goes by embedding the
disruptive group $(\ZZ/\ell)^3$ into $\G$ and then showing that some
$x\in\Hoh^4(\BG,\ZZ)$ restricts to the class $y\in\Hoh^4(\B(\ZZ/\ell)^3,\ZZ)$ found by
Atiyah and Hirzebruch. In other words, they show that $d_{2\ell-1}(x)\neq 0$ in the
Atiyah-Hirzebruch spectral sequence.

The purpose of the present paper is to return to these questions and
give a different type of argument, for certain groups of type $A_n$, relying on representation
theory.

Let $\G$ be a quotient of $\SL_{n}$ by a central subgroup $\mathrm{C}\subseteq\mu_n$. Then,
$\Hoh^4(\BG,\ZZ)\rightarrow\Hoh^4(\BSL_{n},\ZZ)$ is injective, and we can write
$\Hoh^4(\BG,\ZZ)$ as $\ZZ\cdot nc_2(\gamma_n)$, some multiple of $c_2$ of the standard
$n$-dimensional representation of $\SL_{n}$. Note however, that unless $\mathrm{C}$ is
trivial $\gamma_n$ does not descent to a representation of $\G$.

\begin{theorem}\label{thm:main}
    Let $\G$ be either $\SL_{8}/\mu_2$ or $\SL_{9}/\mu_3$,
    viewed as complex affine algebraic groups.
    Then, $\Hoh^4(\BG,\ZZ)=\ZZ\cdot c_2$, while the image of
    $\CH^2(\BG)\rightarrow\Hoh^4(\BG,\ZZ)$ is $2c_2\cdot\ZZ$ (resp. $3 c_2\cdot\ZZ$).
\end{theorem}

Standard arguments then give counterexamples to the mod-torsion Hodge and Tate conjectures
in characteristic zero and the Tate conjecture over $\FF_p$ for $\ell=2,3$ for all but
finitely many primes $p$. In
particular, we obtain a new proof of the following theorem of Pirutka-Yagita (for
$\ell=2,3,5$) and Kameko (for all $\ell$).

\begin{theorem}[\cite{pirutka-yagita}*{Theorem~1.1},\cite{kameko}*{Theorem~1.1}]\label{thm:tate}
    Let $\ell=2,3$. Then, for all primes $p\neq\ell$ there exists a smooth
    projective variety $X$ over $\FF_p$ such that the cycle map
    \begin{equation*}
        \cl^2:\CH^2(X_{\overline{k}})_{\ZZ_{\ell}}\rightarrow\bigcup_U\Hoh^{4}(X_{\overline{k}},\ZZ_{\ell}(i))^U/\mathrm{torsion}
    \end{equation*}
    is not surjective.
\end{theorem}

Our arguments not only differ from those used by Pirutka, Yagita, and Kameko, but they must
differ. Indeed, we show that in
the cases of Theorem~\ref{thm:main} the differential $d_{2\ell-1}(c_2)=0$ in contrast to the
work of the previous authors. Moreover, if $\rho(c_2)$ is the reduction modulo $\ell$ of
$c_2\in\Hoh^4(\BG,\ZZ)$ for $G$ as in the theorem, then all of Milnor's operations $\Q_i$
vanish on $\rho(c_2)$. However, thanks to the Atiyah-Segal completion theorem, some later
differential $d_r$ in the spectral sequence does detect the non-algebraicity of $c_2$;
see Proposition~\ref{prop:diff2}.

The approach we discovered is checkable for $\ell=2,3$ and in theory for
$\SL_{\ell^2}/\mu_\ell$ for any given odd prime $\ell$.
Namely, we use the fact that $\CH^2(\BG)$ is generated by the Chern classes of polynomial
representations of $\G$. We then determine generators for the representation rings of
$\SL_{8}/\mu_2$ and $\SL_{9}/\mu_3$ and compute, using a short \texttt{SAGE}~\cite{sage}
program, the second Chern
classes of all of these representations. The theorem follows immediately from this
computation.

\begin{conjecture}
    If $\ell$ is an odd prime, the image of $\CH^2(\BG)\rightarrow\Hoh^4(\BG,\ZZ)=\ZZ\cdot c_2$ is generated by
    $\ell c_2$ for $\G=\SL_{\ell^2}/\mu_\ell$.
\end{conjecture}

This paper is organized as follows. In Section~\ref{sec:charzero} we prove
Theorem~\ref{thm:main}. In Section~\ref{sec:charp} we construct smooth projective examples
and reduce modulo $p$. Finally, in
Section~\ref{sec:code} we explain the \texttt{SAGE} code we use for computations.

\paragraph{Acknowledgments.} We would like to thank Alena Pirutka sparking our interest in
this problem by giving a stimulating
talk at the Workshop on Chow groups, motives and derived categories at IAS in March 2015,
organized by Totaro and Voisin. We would also like to thank Alena Pirutka, Burt Totaro, and
Masaki Kameko for providing us with many useful comments on an early draft of the paper.
Special thanks go to Kameko and Totaro for pushing us to prove Theorem~\ref{thm:tate} for all primes
$p$.

\section{Examples in characteristic zero}\label{sec:charzero}

In this section we work with affine algebraic groups $\G$ over the complex numbers. Implicitly,
by $\Hoh^*(\BG,\ZZ)$ we mean $\Hoh^*(\BG(\CC),\ZZ)$ and so on. We will use the following
result from Totaro's book, which can also be found in the work of
Esnault-Kahn-Levine-Viehweg.

\begin{theorem}[\cite{totaro-book}*{Theorem
    2.25}, \cite{esnault-kahn-levine-viehweg}*{Lemma C.3}]\label{thm:ch2}
    If $\G$ is an affine algebraic group scheme over a field $k$, then $\CH^2(\BG)$ is generated by
    Chern classes of polynomial representation.
\end{theorem}

Note that the statement in the book is for modulo $\ell$ Chow groups, but that the proof
works integrally for $\CH^2(\BG)$.

We say that the integral Hodge conjecture holds in codimension $2$ for $\BG$ where $\G$ is an affine
algebraic group over $\CC$ if $\Hoh^4(\BG,\ZZ)$ is generated by Chern classes of
representations. This is justified by Theorem~\ref{thm:ch2} and the fact that when it
fails one can construct counterexamples to the integral Hodge conjecture in codimension $2$
on smooth projective complex varieties.

\subsection{Groups of type $A_n$ for small $n$}

It is well-known that for $\G=\SL_n$ or $\G=\GL_n$ the cycle class maps give ring
isomorphisms $\CH^*(\BG)\rightarrow\Hoh^*(\BG,\ZZ)$. See~\cite{totaro-book}*{Chapter 2}. In
other low-dimensional cases, we can
obtain a similar conclusion, at least for $\Hoh^4(\BG,\ZZ)$. Write $\gamma_n$ for the
standard $n$-dimensional representation of $\SL_n$.

\begin{proposition}\label{prop:low}
    For $\G=\PGL_2,\PGL_3,\SL_4/\mu_2,\PGL_4,\PGL_5,\SL_6/\mu_2,\SL_6/\mu_3,\PGL_6$, or $\PGL_7$,
    $\Hoh^4(\BG,\ZZ)$ is generated by Chern classes of polynomial representations.
\end{proposition}

\begin{proof}
    For $\PGL_3$, $\PGL_5$, and $\PGL_7$, this follows from the work of
    Vistoli~\cite{vistoli}, and also from Vezzosi~\cite{vezzosi} for $\PGL_3$. Indeed,
    Vistoli shows more generally that $\CH^*(\BPGL_p)\rightarrow\Hoh^*(\BPGL_p,\ZZ)$ is an
    isomorphism onto the even cohomology for $p$ an odd prime. The case of $\PGL_2\iso\SO_3$
    follows from Pandharipande's calculations in~\cite{pandharipande}. The claim now follows
    from Theorem~\ref{thm:ch2}.

    The group $\Hoh^4(\BPGL_n,\ZZ)\subseteq\Hoh^4(\BSL_n,\ZZ)=\ZZ\cdot c_2$ is generated by
    $2nc_2$ for $n$ even and $nc_2$ for $n$ odd. On the other hand, the adjoint
    representation has $c_2(\mathrm{Ad}_{\PGL_n})=2nc_2$. For details,
    see~\cite{woodward}*{Sections 1,2}. For $n$ even it then follows that
    $\Hoh^4(\BPGL_n,\ZZ)$ is generated by Chern classes, so this takes care of $\PGL_4$ and
    $\PGL_6$.

    One checks, as we will do in detail for $\SL_8/\mu_2$ in the next section, that in the remaining cases we
    have
    \begin{align*}
        \Hoh^4(\BSL_4/\mu_2,\ZZ)&=\ZZ\cdot 2c_2,\\
        \Hoh^4(\BSL_6/\mu_2,\ZZ)&=\ZZ\cdot 4c_2,\\
        \Hoh^4(\BSL_6/\mu_3,\ZZ)&=\ZZ\cdot 3c_2.
    \end{align*}
    In these cases, $c_2(\wedge^2\gamma_4)=2c_2$,
    $c_2(\wedge^2\gamma_6)=4c_2$, $c_2(\wedge^3\gamma_6)=6c_2$, and
    $c_2(\gamma_6^{(2,1)})=33c_2$, where $\gamma_6^{(2,1)}$ is the irreducible
    representation of $\SL_6$ corresponding to the partition $(2,1)$. See
    Section~\ref{sec:code} for details. In particular, we see
    that Chern classes generate $\Hoh^4(\BG,\ZZ)$ in each of the remaining cases. These representations
    all descend to the quotient groups because they are trivial on the given central
    subgroups.
\end{proof}

\subsection{$\BSL_8/\mu_2$}

Let $G=\SL_8/\mu_2$ over $\CC$.

\begin{theorem}\label{thm:p28}
    The pullback map $\Hoh^4(\BG,\ZZ)\rightarrow\Hoh^4(\BSL_8,\ZZ)=\ZZ\cdot c_2$ is an
    isomorphism, while the image of the cycle class map $\cl^2$ is $\ZZ\cdot 2c_2$ in
    $\Hoh^4(\BG,\ZZ)$. In particular,
    the integral Hodge conjecture fails in codimension $2$ for $\BG=\BSL_{8}/\mu_2$.
\end{theorem}

\begin{proof}
    Consider the $\Eoh_2$-page of the Leray-Serre spectral sequence for the fibration
    $\B\mu_2\rightarrow\BSL_{8}\rightarrow\BG$ displayed in Figure~\ref{fig:lerayserre}.

    \begin{figure}[h]
        \centering
        \begin{equation*}
            \xymatrix@R=8px@C=8px{
                (\ZZ/2)\cdot x^2 \ar^{d_3}[ddrrr] \ar@/^20px/^{d_5}[ddddrrrrr]\\
                0\\ (\ZZ/2)\cdot x
            \ar_{d_3}[ddrrr] & 0 & \ZZ/2 \ar@/_5px/^{d_3}[ddrrr]& \ZZ/2 \\ 0 \\ \ZZ & 0 & 0  &
        \Hoh^3(\BG, \ZZ) & \Hoh^4(\BG, \ZZ) & \Hoh^5(\BG,\ZZ) }
        \end{equation*}
        \caption{The $\Eoh_2$-page of the Leray-Serre spectral sequence associated with
            $\B\mu_2\rightarrow\BSL_{8}\rightarrow\BSL_8/\mu_2$.}
        \label{fig:lerayserre}
    \end{figure}

    The differential $d_3^{0,2}$ is an isomorphism as $\Hoh^3(\BG,\ZZ)\iso\ZZ/2$, and we
    know that $\Hoh^*(\BSL_{8},\ZZ)$ is torsion-free. Using the Leibniz rule,
    we see that $d_3(x^2)=0$. Moreover, $\Hoh^5(\BG,\ZZ)\iso\ZZ/4$ by the
    proof of~\cite{aw3}*{Proposition 6.2}. It follows that $d_3^{2,2}$ is non-zero, as is
    $d_5^{4,0}$, for otherwise there would be torsion in $\Hoh^5(\BSL_8,\ZZ)$. This shows
    that the edge map $\Hoh^4(\BG,\ZZ)\rightarrow\Hoh^4(\BSL_{8},\ZZ)$ is an isomorphism.

    It follows that given a cycle $x\in\CH^2(\BG)$, we can write $\cl^2(x)=n_x c_2$, where
    $c_2\in\Hoh^4(\BSL_{8},\ZZ)$ is the Chern class of the standard representation $\gamma_8$ of
    $\SL_{8}$. We will see that $n_x$ is always even.

    By Theorem~\ref{thm:ch2},
    to finish the proof it is enough to compute the second Chern classes of generating
    representations of the representation ring $\Roh[\G]$. Let $L_1,\ldots,L_8$ generate the
    weight lattice of $\SL_{8}$, with fundamental weights $\alpha_i=L_1+\cdots+L_i$ for $1\leq
    i\leq 7$. An irreducible representation with highest weight vector $\lambda_1L_1+\cdots+\lambda_8L_8$ of
    $\SL_{8}$ descends to a representation of $\G$ if and only if
    $\sum_i\lambda_i$ is even. Table~\ref{tab:c2}
    gives the generators of the weight lattice of $\G$, the associated Young
    diagrams, and the numbers $n_\lambda$ such that
    $c_2(\gamma_8^\lambda)=n_\lambda c_2(\gamma_8)=n_\lambda c_2$ in
    $\Hoh^4(\BSL_{8},\ZZ)$. See Section~\ref{sec:code} for how the table was computed.
    By inspection, these Chern classes are all even with greatest common
    divisor $2$, from which the theorem follows.
\end{proof}

\begin{center}
    \begin{table}[hl]
    \begin{tabular}{| M{1.5cm} | M{3cm} | M{1.5cm} ||| M{1.5cm} | M{3cm} | M{1.5cm} |}
        \hline
        Weight & $\lambda$ & $n_\lambda$ &Weight & $\lambda$ & $n_\lambda$ \\
        \hline
        $2\alpha_1$ & $(2)$ & $16$ & $\alpha_1+\alpha_3$&$(2,1,1)$& $156$\\
        $\alpha_2$&$(1,1)$& $6$ &$\alpha_1+\alpha_5$&$(2,1,1,1,1)$& $170$\\
        $2\alpha_3$&$(2,2,2)$& $700$ &$\alpha_1+\alpha_7$&$(2,1,1,1,1,1,1)$& $16$\\
        $\alpha_4$&$(1,1,1,1)$& $20$ & $\alpha_3+\alpha_5$&$(2,2,2,1,1)$& $1344$\\
        $2\alpha_5$&$(2,2,2,2,2)$& $700$ &      $\alpha_3+\alpha_7$&$(2,2,2,1,1,1,1)$& 170\\
        $\alpha_6$&$(1,1,1,1,1,1)$& $6$      & $\alpha_5+\alpha_7$&$(2,2,2,2,2,1,1)$& 156\\
        $2\alpha_7$&$(2,2,2,2,2,2,2)$& 10 &&&\\
        \hline
    \end{tabular}
    \caption{The second Chern classes of generators of $\Roh[\SL_8/\mu_2]$.}
    \label{tab:c2}
    \end{table}
\end{center}

The class $c_2\in\Hoh^4(\BG,\ZZ)$ is a non-torsion counterexample to the integral Hodge
conjecture, which we see directly from representation theory. However, it is nevertheless
important to know that some cohomology operation tells us that $c_2$ cannot be a cycle
class. This is critical when we want to pass from $\BG$ to a smooth projective variety,
either over $\CC$ or over a finite field.

In all previous group-theoretic counterexamples at the prime $\ell=2$, such as those of
Atiyah-Hirzebruch, Totaro, Pirutka-Yagita, or Kameko, one has $d_3(x)=\Sq^3_\ZZ(x)\neq 0$, where $d_3=\Sq^3$ is the
differential in the Atiyah-Hirzebruch spectral sequence converging to
$\KU^*(\BG)$, where $\KU$ is complex topological $K$-theory. This is
proved by showing more strongly that the Milnor operation $\Q_1$ does not vanish on
$\rho(x)$.
In our case, we have $d_3(c_2)=0$, and moreover $\Q_i(\rho(c_2))=0$ for all Milnor operations
$\Q_i$, where $\rho:\Hoh^*(\BG,\ZZ)\rightarrow\Hoh^*(\BG,\ZZ/2)$.

Recall that the Milnor operations are stable cohomology operations
$$\Q_i:\Hoh^m(X,\ZZ/\ell)\rightarrow\Hoh^{m+2\ell^i-1}(X,\ZZ/\ell).$$ These were defined in
motivic cohomology as well by Voevodsky~\cite{voevodsky-reduced}*{Section 9}, where we refer the reader for an overview. The key
point is that $\Q_i$ vanishes on the image of the cycle class map.

\begin{theorem}\label{thm:milnor2}
    Let $c_2\in\Hoh^4(\BG,\ZZ)$ be the generator, where $\G=\SL_8/\mu_2$, and let
    $x=\rho(c_2)$. Then, $\Q_i(x)=0$ for all $i\geq 0$.
\end{theorem}

\begin{proof}
    The operation $\Q_0$ is the Bockstein operation, so it follows immediately that
    $\Q_0(x)=0$ since $x$ is the reduction of an integral class. The argument that
    $\Q_1(x)=0$ is more subtle, as it involves the Postnikov tower for $\BG$. Consider the
    commutative diagram of Postnikov sections:
    \begin{equation*}
        \xymatrix{
            K(\ZZ,6)\ar[r]\ar[d]^{=}    &   \tau_{\leq 6}\BSL_8\ar[r]\ar[d] & K(\ZZ,4)\ar[r]^{k_6}\ar[d]    & K(\ZZ,7)\ar[d]^{=}\\
            K(\ZZ,6)\ar[r]              &   \tau_{\leq 6}\BG\ar[r]          & K(\ZZ/2)\times K(\ZZ,4)\ar[r]   &    K(\ZZ,7).
        }
    \end{equation*}
    Recall that for a connected topological space $X$ and $n\geq 0$, the
    Postnikov section $\tau_{\leq n}X$ is a space with $\pi_i\tau_{\leq n}X=0$
    for $i>n$ equipped with an $n+1$-equivalence $X\rightarrow\tau_{\leq n}X$.
    Since $\Hoh^7(K(\ZZ,4),\ZZ)=\ZZ/2$, and because the cohomology of $\BSL_8$ is
    non-torsion, we see that $k_6$ must exactly be the non-zero cohomology class. But, this
    class is exactly the integral operation $\Sq^3_{\ZZ}$ of the fundamental class of
    $K(\ZZ,4)$. See~\cite{atiyah-hirzebruch}*{Section 7} for the relationship
    between the operation $\Sq^3_{\ZZ}$ and $d_3$.
    It follows from the commutative diagram that $d_3(c_2)=\Sq^3_\ZZ(c_2)=0$ since
    we construct $\tau_{\leq 6}\BG$ exactly using this class as the $k$-invariant. Now,
    $\Q_1(x)=\Q_1(\rho(c_2))=\rho(\Sq^3_\ZZ(c_2))=0$.

    In general, one has $\Q_i=[\P^{\ell^{i-1}},\Q_{i-1}]$, the commutator in the Steenrod
    algebra.  The cohomology operations $\P^i$ vanish on $\Hoh^4(X,\ZZ/2)$ for $i>2$ for degree
    reasons for any space $X$. It follows it is enough to show that $\Q_2=0$ because after
    that all higher Milnor operations will vanish for degree reasons on $\P^i$ and the
    commutator relation. Since $\Q_2=[\P^1,\Q_1]=[\Sq^2,\Q_1]$ at the prime $2$, we have $\Q_2(x)=\Q_1(\Sq^2(x))$. Now,
    $\Sq^2(x)=x^2$ since $x$ has degree $4$, and we can use that $\Q_1=[\P^1,\Q_0]$ to see that
    \begin{equation*}
        \Q_2(x)=\Q_1(\Sq^2(x))=\Q_1(x^2)=\Q_0(\P^1(x^2))=\Q_0(\Sq^1(x)\Sq^1(x))
    \end{equation*}
    by the Cartan formula and the fact that $\Q_0(x^2)=0$ since $x$ is the reduction of a
    an integral class. For the same reason, $\Sq^1(x)=0$, and hence $\P^2(x)=0$.
    All notation and relations are as in~\cite{voevodsky-reduced}.
\end{proof}

We need to know that nevertheless some cohomology operation detects the non-algebraicity of
$c_2$. This is provided by the next proposition.

\begin{proposition}\label{prop:diff2}
    For some odd integer $r>3$, the differential $d_r(c_2)\neq 0$ in the
    Atiyah-Hirzebruch sequence for $\BG=\BSL_8/\mu_2$.
\end{proposition}

\begin{proof}
    We use the Atiyah-Segal completion theorem~\cite{atiyah-segal-completion},
    which says that $\KU^*(\BG)$ is isomorphic
    to the completion $\mathrm{R}[\G]_{\widehat{I}}$ of $\mathrm{R}[\G]$ at the
    augmentation ideal, where $\KU$ denotes complex topological $K$-theory. Strictly speaking, the completion theorem is about compact Lie
    groups, but in this case there is no difference either in terms of representation theory
    or of $K$-theory in working with $\G=\SL_8/\mu_2$ versus $\mathrm{SU_8}/\mu_2$.
    Both $\mathrm{R}[\G]_{\widehat{I}}$ and $\KU^0(\BG)$ are natural filtered, the former
    by the powers of the augmentation ideal and the latter as the abutment the Atiyah-Hirzebruch spectral
    sequence. The Atiyah-Segal isomorphism is compatible with this filtration, which means
    in particular that the image of $\KU^*(\BG)\rightarrow\Hoh^4(\BG,\ZZ)$ is the same as the
    image of $c_2:I^2/I^3$, since $\Hoh^2(\BG,\ZZ)=0$. But, we have already found that this
    map is not surjective, so that it follows that $c_2$ must not be permanent in the
    Atiyah-Hirzebruch spectral sequence. The proposition follows.
\end{proof}

\subsection{$\BSL_9/\mu_3$}

We take a similar approach for $\G=\BSL_9/\mu_3$ over $\CC$.

\begin{theorem}\label{thm:39}
    The pullback map $\Hoh^4(\BG,\ZZ)\rightarrow\Hoh^4(\BSL_9,\ZZ)=\ZZ\cdot c_2$ is an
    isomorphism, while the image of the cycle class map $\cl^2$ is $\ZZ\cdot 3c_2$ in
    $\Hoh^4(\BG,\ZZ)$. In particular, the integral Hodge conjecture fails in codimension $2$ for
    $\BG=\BSL_9/\mu_3$.
\end{theorem}

\begin{proposition}\label{prop:pullback}
    For any odd prime $\ell$, the pullback map
    $\Hoh^4(\BG,\ZZ)\rightarrow\Hoh^4(\BSL_{\ell^2},\ZZ)=\ZZ\cdot c_2$ is an isomorphism.
\end{proposition}

\begin{proof}
    As above, we consider the $\Eoh_2$-page of the Leray-Serre spectral sequence
    $\B\mu_\ell\rightarrow\BSL_{\ell^2}\rightarrow\BG$ displayed in Figure~\ref{fig:lerayserre39}.

    \begin{figure}[h]
        \centering
        \begin{equation*}
            \xymatrix@R=8px@C=8px{
                (\ZZ/\ell)\cdot x^2 \ar^{d_3}[ddrrr] \ar@/^20px/^{d_5}[ddddrrrrr]\\
                0\\ (\ZZ/\ell)\cdot x
            \ar_{d_3}[ddrrr] & 0 & \ZZ/\ell \ar@/_5px/^{d_3}[ddrrr]& \ZZ/\ell \\ 0 \\ \ZZ & 0 & 0  &
        \Hoh^3(\BG, \ZZ) & \Hoh^4(\BG, \ZZ) & \Hoh^5(\BG,\ZZ) }
        \end{equation*}
        \caption{The $\Eoh_2$-page of the Leray-Serre spectral sequence associated with
            $\B\mu_\ell\rightarrow\BSL_{\ell^2}\rightarrow\BSL_{\ell^2}/\mu_\ell$.}
        \label{fig:lerayserre39}
    \end{figure}

    The main outside input this time from~\cite{aw3}*{Proposition 6.2} is that
    $\Hoh^5(\BG,\ZZ)\iso\ZZ/\ell$. However, in this case, $d_3^{0,2}$ is an isomorphism as is
    $d_3^{0,4}$, by the Leibniz rule. It follows that $d_3^{2,2}$ is also an isomorphism,
    from which it follows that the edge map $\Hoh^4(\BG,\ZZ)\rightarrow\Hoh^4(\BSL_9,\ZZ)$
    is an isomorphism.
\end{proof}

\begin{proof}[Proof of Theorem~\ref{thm:39}]
    As above, it suffices to compute the Chern classes of generators of the representation
    ring of $\G=\SL_9/\mu_3$. This time there are $23$ generators. The results are displayed
    in Table~\ref{tab:c3}. The theorem follows by inspection.
\end{proof}

\begin{center}
    \begin{table}[hl]
    \begin{tabular}{| M{1.5cm} | M{3cm} | M{1.5cm} ||| M{1.5cm} | M{3cm} | M{1.5cm} |}
        \hline
        Weight & $\lambda$ & $n_\lambda$ &Weight & $\lambda$ & $n_\lambda$ \\
        \hline
        $3\alpha_1$ &   $(3)$   &   $165$    &   $2\alpha_1+\alpha_7$   & $(3,1,1,1,1,1,1)$   &   $693$\\
        $3\alpha_2$ &   $(3,3)$ &   $3465$  &   $\alpha_2+\alpha_4$     &   $(2,2,1,1)$ & $1701$\\
        $\alpha_3$   &   $(1,1,1)$   &   $21$    &   $\alpha_2+\alpha_7$   &   $(2,2,1,1,1,1,1)$   & $486$\\
        $3\alpha_4$  &   $(3,3,3,3)$ &   $116424$ &   $2\alpha_2+\alpha_5$   & $(3,3,1,1,1)$ & $37125$\\
        $3\alpha_5$ &   $(3,3,3,3,3)$   &   $116424$    &   $2\alpha_2+\alpha_8$    & $(3,3,1,1,1,1,1,1)$ &   $2541$\\
        $\alpha_6$  &   $(1,1,1,1,1,1)$ &   $21$    &   $\alpha_4+\alpha_5$ & $(2,2,2,2,1)$   &   $5292$\\
        $3\alpha_7$ &   $(3,3,3,3,3,3,3)$ & $3465$  &   $\alpha_4+\alpha_8$ & $(2,2,2,2,1,1,1,1)$ & $420$\\
        $3\alpha_8$ & $(3,3,3,3,3,3,3,3)$ & $66$ & $2\alpha_4+\alpha_7$ & $(3,3,3,3,1,1,1)$ & $117810$\\
        $\alpha_1+\alpha_2$ & $(2,1)$ & $78$ & $\alpha_5+\alpha_7$ & $(2,2,2,2,2,1,1)$ & $1701$\\
        $\alpha_1+\alpha_5$ & $(2,1,1,1,1)$ & $420$ & $2\alpha_5+\alpha_8$ & $(3,3,3,3,3,1,1,1)$ & $29106$\\
        $\alpha_1+\alpha_8$ & $(2,1,1,1,1,1,1,1)$ & $18$ & $\alpha_7+\alpha_8$ & $(2,2,2,2,2,2,2,1)$ & $78$\\
        $2\alpha_1+\alpha_4$ & $(3,1,1,1)$ & $2541$ & & &\\
        \hline
    \end{tabular}
    \caption{The second Chern classes of generators of $\Roh[\SL_9/\mu_3]$.}
    \label{tab:c3}
    \end{table}
\end{center}

We include without proof the following results, as the proofs are minor modifications of the
proofs of Theorem~\ref{thm:milnor2} and Proposition~\ref{prop:diff2}.

\begin{theorem}\label{thm:milnor3}
    Let $c_2\in\Hoh^4(\BG,\ZZ)$ be the generator, where $\G=\SL_9/\mu_3$, and let
    $x=\rho(c_2)\in\Hoh^4(\BG,\ZZ/3)$. Then, $\Q_i(x)=0$ for all $i\geq 0$.
\end{theorem}

\begin{proposition}\label{prop:diff3}
    For some odd integer $r>5$, $d_r(c_2)\neq 0$ in the cohomology of $\BG=\BSL_9/\mu_3$.
\end{proposition}

\subsection{$\BSL_{\ell^2}/\mu_\ell$}

It is clear that we cannot simply go on computing representations to get a general result.
Nevertheless, for odd primes $\ell>3$ we expect by Proposition~\ref{prop:pullback} that a
result similar to Theorem~\ref{thm:39} holds.

\begin{conjecture}\label{conj:ell}
    Fix an odd prime $\ell$. Let $\G=\SL_{\ell^2}/\mu_\ell$ over $\CC$.
    Then, the image of the cycle class map $\cl^2$ is the subgroup $\ZZ\cdot\ell
    c_2\subseteq\ZZ\cdot c_2=\Hoh^4(\BG,\ZZ)$.
\end{conjecture}

We intend to return to this question in future work with Ben Williams on the groups
$\BSL_{\ell^2}/\mu_\ell$ and the topological period-index problem.

\section{Counterexamples to the integral Hodge and Tate conjectures}\label{sec:charp}

There is a standard way to pass from counterexamples to the integral Hodge conjecture in the
cohomology of $\BG$ to actual counterexamples in the cohomology of smooth projective complex
varieties. There is also a method for defining these integrally and reducing modulo $p$ to
obtain characteristic $p$ counterexamples to the integral Tate conjecture.

We give two modifications to these methods. The first is that we use Poonen's
work on Bertini theorems over finite fields~\cite{poonen} to obtain results
over \emph{every} finite field different from $\ell$. The second is that we
make an argument using Thomason's paper on \'etale cohomology~\cite{thomason} to
show that certain cohomology operations in singular cohomology with
$\ZZ_\ell$-coefficients come from functorial cohomology operations in
$\ell$-adic cohomology. This step is necessary to ensure that upon reducing to
finite characteristic we still have access to an operation that detects the
non-algebraicity of the class $c_2$.

\begin{theorem}
    Fix a prime $\ell=2,3$. Let $\BG=\BSL_8/\mu_2$ or $\BSL_9/\mu_3$ according to whether
    $\ell=2$ or $\ell=3$. There exists a smooth projective complex algebraic variety
    $X$ and a map $f:X\rightarrow\BG$ such that $x=f^*(c_2)$ is a non-torsion Hodge class in
    $\Hoh^4(X,\ZZ)$ that is not a cycle class, while $\ell x$ is a cycle class. Moreover,
    $d_{2\ell-1}(x)=0$.
\end{theorem}

\begin{proof}
    We leave the proof to the reader as an application of Theorems~\ref{thm:milnor2} and
    Theorem~\ref{thm:milnor3} and an easy modification of the proof of the next theorem.
\end{proof}

The reader should compare the theorem to the original result of Atiyah and Hirzebruch where
the non-algebraicity of $x$ is exactly detected by the fact that $d_{2\ell-1}(x)\neq 0$.

\begin{theorem}[\cite{pirutka-yagita}*{Theorem~1.1},\cite{kameko}*{Theorem~1.1}]
    Let $\ell=2,3$. Then, for all primes $p\neq\ell$ there exists a smooth
    projective variety $X$ over $\FF_p$ such that the cycle class map
    \begin{equation*}
        \cl^2:\CH^2(X_{\overline{k}})_{\ZZ_{\ell}}\rightarrow\bigcup_U\Hoh^{4}(X_{\overline{k}},\ZZ_{\ell}(i))^U/\mathrm{torsion}
    \end{equation*}
    is not surjective, where $U$ varies over the open subgroups of $\Gal(\overline{k}/k)$. Moreover, the Milnor operations $\Q_i$ vanish on
    $\Hoh^4(X,\ZZ/\ell(2))$.
\end{theorem}

\begin{proof}
    Let $\G=\SL_{8,\CC}/\mu_2$ or $\G=\SL_{9,\CC}/\mu_3$, according to the choice of $\ell$.
    Let $r>3$ be the unique (necessarily odd) integer such that $d_r(c_2)\neq 0$ in $\Hoh^*(\BG,\ZZ)$, where
    $d_r$ is the differential in the Atiyah-Hirzebruch spectral sequence computing
    $\KU^*(\BG)$.

    Now, let $\G_{\ZZ[1/\ell]}$ denote either $\SL_{8,\ZZ[1/2]}/\mu_2$ or
    $\SL_{9,\ZZ[1/3]}/\mu_3$, smooth
    group-schemes over appropriate localizations of $\ZZ$.
    The arguments in~\cite{pirutka-yagita}*{Section 4} show how to construct a smooth
    quasi-projective $t$-dimensional scheme $Y\subseteq\PP^N_{\ZZ[1/\ell]}$
    over $\Spec\ZZ[1/\ell]$
    such that there is a map
    $Y\rightarrow\B\mathds{G}_{m,\ZZ[1/\ell]}\times\BG_{\ZZ[1\ell]}$ having the
    following properties:
    \begin{enumerate}
        \item   the induced map $Y_\CC\rightarrow\B\mathds{G}_{m,\CC}\times\BG_{\CC}$ is an
            $r+4$-equivalence;
        \item   the complement $W$ of $Y$ in the closure of $Y$ in
            $\PP^N_{\ZZ[1/\ell]}$ has large codimension $c$ (to be specified below).
    \end{enumerate}
    Let $s=\Spec\FF_p$ be a closed point of $\Spec\ZZ[1/\ell]$ (so $p\neq\ell$) and $\overline{s}$ the associated geometric point.
    Write $Y_s$ for the closed fiber of $Y$ over $s$.

    We claim that there exists a complete intersection
    $V_s\subseteq\PP^N_{\FF_p}$ of codimension $t-c+1$ such that $V_s\cap Y_s$ is a smooth
    \emph{projective} variety. It suffices to show that there is a
    hypersurface $H_s$ such that $H_s\cap Y_s$ is smooth and $\dim H_s\cap
    W_s=\dim W_s-1$, for then we can simply continue to cut down with
    hypersurfaces until the boundary is empty. We use
    Poonen's Bertini theorem with Taylor coefficients~\cite{poonen}*{Theorem~1.2}.
    Indeed, if we let $T_s\subseteq W_s$ be a finite subscheme containing at
    least one closed point of each irreducible geometric component of $W_s$,
    then Poonen's theorem implies that there is a positive density of
    hypersurfaces $H_s:f=0$ such that $H_s\cap Y_s$ is smooth of dimension $\dim
    Y_s-1$ and $f$ does not vanish at any point of $T_s$. In particular, it
    follows that $W_s\cap H_s$ must have dimension (at most) $1$ less than the
    dimension of $W_s$, as desired. Iterating this argument, we find $V_s$.

    Let us prove that $X_s=V_s\cap Y_s$ has the required properties. Since
    $V_s$ is a complete intersection, it lifts to characteristic zero, i.e., to
    flat scheme $V$ over the $p$-adics $\ZZ_p$. Moreover, $X=V\cap Y_{\ZZ_p}$ is
    a smooth projective scheme over the $p$-adic integers.
    By Hamm's quasi-projective Lefschetz hyperplane
    theorem~\cite{goresky-macpherson}*{Section II.1.2},
    \begin{equation*}
        \Hoh^i_{\et}(\B\mathds{G}_{m,\CC}\times\BG_\CC,\ZZ_\ell(j))\iso\Hoh^i_{\et}(Y_{\CC},\ZZ_\ell(j))\iso\Hoh^i_{\et}(X_\CC,\ZZ_\ell(j))
    \end{equation*}
    for $i\leq\dim(X)-2=(c-1)-2=c-3$. By choosing $Y$ so that $r+4\leq c-3$ (or
    $r+7\leq c$), one finds that the inequality will hold for $i\leq r+4$. That
    it is possible to find such a $Y$ is proven in~\cite{totaro}*{Remark~1.4}.

    Using the invariance of \'etale cohomology for extensions of
    algebraically closed fields, and proper base change, there are maps
    \begin{equation*}
        \Hoh^i_{\et}(\B\mathds{G}_{m,\CC}\times\BG_\CC,\ZZ_\ell(j))\iso\Hoh^i_{\et}(Y_{\CC},\ZZ_\ell(j))\iso\Hoh^i(X_\CC,\ZZ_{\ell}(i))\rightarrow\Hoh^i_{\et}(X_{\overline{s}},\ZZ_\ell(j))
    \end{equation*}
    which are isomorphisms for
    for $i\leq r+4$.
    Using the comparison isomorphism between $\ell$-adic cohomology and
    singular cohomology for smooth complex schemes, we see that there is a class
    $c_2\in\Hoh^4_{\et}(X_{\overline{s}},\ZZ_\ell(2))$ corresponding to $c_2\in\Hoh^4(\BG,\ZZ_\ell)$.
    Note that by looking at the Chern classes of representations of $\G$ we see that $\ell
    c_2$ is actually defined (by abuse of notation) in $\Hoh^4_{\et}(X,\ZZ_\ell(2))$.
    Hence, $c_2$ defines a $U$-invariant class
    in $\Hoh^4_{\et}(X_{\overline{s}},\ZZ_\ell(2))$ for some open subgroup
    $U\subseteq\mathrm{Gal}(\overline{\FF}_p,\FF_p)$.

    Now, it suffices to see that $c_2$ cannot be in the image of the cycle
    class map. We do not know of a source for the full $\ell$-adic cohomology
    operations, which would say something about $d_r$ for $r>3$, so we argue as
    follows. If $U$ is a smooth complex quasi-projective variety, then
    \begin{equation}\label{eq:comparison}
        \Hoh^i(U,\ZZ_\ell)\iso\Hoh^i_{\et}(U,\ZZ_\ell),
    \end{equation}
    where the first group is singular cohomology with coefficients in
    $\ZZ_\ell$, and the second group is $\ell$-adic cohomology, defined as
    usual by
    \begin{equation*}
        \Hoh^i_{\et}(U,\ZZ_\ell)=\lim_\nu\Hoh^i_{\et}(U,\ZZ/\ell^\nu).
    \end{equation*}
    In order to prove this, note that there are functorial comparison
    isomorphisms $\Hoh^i(U,\ZZ/\ell^\nu)\iso\Hoh^i_{\et}(U,\ZZ/\ell^\nu)$,
    where again the first term is singular cohomology with coefficients in
    $\ZZ/\ell^\nu$. Hence, to prove that~\eqref{eq:comparison} is an
    isomorphism, it suffices to show that the natural map
    \begin{equation*}
        \Hoh^i(U,\ZZ_\ell)\rightarrow\lim_\nu\Hoh^i(U,\ZZ/\ell^\nu)
    \end{equation*}
    is an isomorphism for any space $U$ having the homotopy type of a finite CW
    complex. This is possible to prove by hand, but instead we refer to Bousfield
    and Kan~\cite{bousfield-kan}*{Theorem~IX.3.1} for the Milnor exact sequence
    \begin{equation*}
        0\rightarrow{\lim_\nu}^1\Hoh^{i-1}(U,\ZZ/\ell^\nu)\rightarrow\Hoh^i(U,\ZZ_\ell)\rightarrow\lim_\nu\Hoh^i(U,\ZZ/\ell^nu)\rightarrow 0.
    \end{equation*}
    Since $U$ has the homotopy type of a finite CW complex, the cohomology groups
    $\Hoh^{i-1}(U,\ZZ/\ell^\nu)$ are finitely generated, so the Mittag-Leffler
    condition is automatically satisfied, and the ${\lim}^1$ term vanishes. So,
    we have proved that~\eqref{eq:comparison} is an isomorphism for smooth
    complex quasi-projective varieties $U$.

    Let $\K(U)[\beta^{-1}]_{\ell}$ denote the $\ell$-adic completion of
    $\K/\ell^\nu(U)[\beta^{-1}]$, the $\ell$-adic Bott-inverted $K$-theory of
    Thomason~\cite{thomason}*{Appendix~A}. Typically, the homotopy groups of
    this spectrum are written $\K_*(U)[\beta^{-1}]_\ell$. If $U$ is a quasi-projective scheme
    over $\CC$, this is precisely the $\ell$-complete (periodic) complex $K$-theory
    spectrum of $U$ by~\cite{thomason}*{Theorem~4.11}:
    $\K_m(U)[\beta^{-1}]_\ell\iso\KU^{-m}(U)_\ell$.
    However, $\K(U)[\beta^{-1}]$ is defined functorially over any base field, and there is a spectral sequence
    \begin{equation*}
        \Eoh_2^{p,q}=\Hoh^p_{\et}(U,\ZZ_\ell(\tfrac{q}{2}))\Rightarrow\K_{q-p}(U)[\beta^{-1}]_{\ell},
    \end{equation*}
    which converges in good cases, for instance in the presence of uniform
    bounds on the $\ell$-adic \'etale cohomological dimensions of the residue
    fields of $U$. The $r$th differential $d_r$ has degree $(r,r-1)$. (Implicitly, if $q$ is odd, the
    $\Eoh_2^{p,q}$ entry is zero.) See Thomason~\cite{thomason}*{Theorem~4.1}.
    If $U$ is a $\CC$-scheme, then the comparison between $\ell$-adic
    Bott-inverted algebraic $K$-theory and periodic complex $K$-theory is
    compatible (via the Atiyah-Hirzebruch spectral sequence)
    with the comparison between singular and \'etale $\ell$-adic
    cohomology. In particular, the differential $d_r$
    leaving the $\Eoh^{4,4}_r$-term of the spectral sequence gives a cohomology
    operation on a subquotient of $\ell$-adic cohomology which is compatible with base change.
    
    In the Atiyah-Hirzebruch spectral sequence for $\BG_\CC$ we have that
    $d_{r}(c_2)\neq 0$. Using the compatibility established above, we see that
    $d_r(c_2)\neq 0$ in a subquotient of
    $\Hoh^{r+4}(X_{\overline{s}},\ZZ_\ell(2+\frac{r-1}{2}))$. On the other hand, this
    operation $d_r$ vanishes on the $\ell$-adic cohomology of
    $\BGL_{n,\overline{\FF}_p}$, since $r$ must be
    odd and $\BGL_{n,\overline{\FF}_p}$ does not support any odd cohomology. As
    $\CH^2(X_{\overline{s}})$ is
    generated by second Chern classes of representations, we see that $c_2$
    cannot be in the image of the cycle class map.
\end{proof}

\section{Appendix: \texttt{SAGE} code}\label{sec:code}

The following simple \texttt{SAGE}~\cite{sage} script computes the second Chern classes,
with help under the hood from
\texttt{Singular}~\cite{singular} for multivariable polynomials.
The function \texttt{trim} serves to truncate polynomials to make the calculation faster by ignoring
higher degree terms. The function \texttt{c2} returns $n_p$ where
$c_2(\gamma_n^p)=n_pc_2(\gamma_n)$ for a partition $p$.  The input $p$ is of the form
$(p_1,\ldots,p_k)$.

The main work in the code is done by the built-in \texttt{SAGE} function
\texttt{SemiStandardTableaux}, which returns all semistandard tableaux on a given partition.
As explained in Fulton's book~\cite{fulton-tableaux}, these index a basis of
$\gamma^\lambda$, which is then used to compute the Chern classes via the splitting
principle. It seems easier to compute in the cohomology of classifying space of the maximal
torus of $\GL_8$ or $\GL_9$, so the code just cancels out any $c_1$ contributions.
We do this by subtracting off $kc_1^2$
where $k$ is the coefficient of $x_1^2$ in the resulting polynomial. The function then
returns the coefficient of $x_1x_2$.


\begin{lstlisting}
    def trim(f,n):
        g=0
        for m in f.monomials():
            if m.degree()<=n:
            g=g+f.monomial_coefficient(m)*m
        return g

    def c2(n,p):
        R=PolynomialRing(ZZ,['x'+repr(i) for i in range(1,n+1)])
        q=1
        for t in SemistandardTableaux(p,max_entry=n).list():
            m=0
            for r in t:
                for c in r:
                    m=m+eval('x'+repr(c))
                    q=q*(1+m)
                    q=trim(q,2)
        c1=(sum(v for v in R.gens()))^2
        r=q-q.monomial_coefficient(x1^2)*c1
\end{lstlisting}

This code will be maintained on the author's \texttt{github} page under the
name \texttt{c2.sage} at
\texttt{gist.github.com/benjaminantieau} and on the author's webpage.

\end{document}